\documentclass[a4paper,10pt]{amsart}
\usepackage[latin1]{inputenc}
\usepackage{a4wide}
\usepackage{amsfonts}
\usepackage{amsmath}
\usepackage{amssymb}
\usepackage{amsthm}

\newtheorem{proposition}{Proposition}
\newtheorem {lemma}{Lemma}
\newtheorem {theorem}{Theorem}
\newtheorem {Corollary}{Corollary}

\numberwithin{equation}{section}
\numberwithin{equation}{section}
\numberwithin{equation}{section}
\begin{document}
\title{Spectral Analysis of diffusions with jump boundary}
\author{Martin Kolb}
\address{Department of Statistics, Oxford University, 1 South Parks Road, Oxford, OX1 3TG, UK}
\email{kolb@stats.ox.ac.uk}
\author{Achim W\"ubker}
\address{Fachbereich Mathematik, Universit\"at Osnabr\"uck, Albrechtstraße 28a, 49076 Osnabr\"uck, Germany}
\email{awuebker@Uni-Osnabrueck.de}
\thanks{A.W. is supported by Deutsche Forschungsgemeinschaft}

\begin{abstract}
In this paper we consider one-dimensional diffusions with constant coefficients in a finite interval with jump boundary and a certain deterministic jump distribution. We use coupling methods in order to identify the spectral gap in the case of a large drift and prove that that there is a threshold drift above which the bottom of the spectrum no longer depends on the drift. As a Corollary to our result we are able to answer two questions concerning elliptic eigenvalue problems with non-local boundary conditions formulated previously by Iddo Ben-Ari and Ross Pinsky. 
\end{abstract}
\maketitle
\section{Introduction and Notation}
This article investigates so called diffusions with jump boundary, which in recent years gave rise to several interesting results (see e.g. \cite{GK1}, \cite{GK2}, \cite{K}, \cite{BaP1}, \cite{BaP2}, \cite{LL} and \cite{LP}). The process itself can be easily described. Consider a diffusion process with initial value $x \in D$ in an open domain $D \subset \mathbb{R}^d$ which we assume to have a $C^{2,\alpha}$-boundary for convenience. When hitting the boundary $\partial D$ of $D$, the diffusion gets redistributed in $D$ according to the jump distribution $\nu$, runs again until it hits the boundary, gets redistributed and repeats this behavior forever. It is heuristically obvious that this process converges in total variation towards its invariant measure, but good or even precise estimates for the rate of convergence are not easy to obtain. One of the main difficulties consists in the non-reversibility of this process. Due to this non-reversibility it is not possible to obtain estimates on the eigenvalues via well-known variational principles. 

If one specializes to the case of a one-dimensional Brownian motion in an interval $(a,b)$ with jump boundary, the exact calculation of the rate of convergence is possible. It turns out that the rate of convergence coincides with the second eigenvalue of the Dirichlet Laplacian $-\frac{1}{2}\frac{d^2}{dx^2}$ in $(a,b)$ independent of the choice of the jump distribution $\nu$. This has been shown in \cite{LLR} using Fourier analytic techniques in combination with a result obtained in \cite{BaP1} and in \cite{KW} using a probabilistic approach. 

Concerning the general diffusion process $(X_t)_{t \geq 0}$ with jump boundary, four open questions are listed in \cite{BaP1}; Question 4 concerning the continuous dependence of the spectral gap on the jump distribution has been answered affirmatively in \cite{KW}. From the remaining three, Question 1 and Question 2 are highly correlated. 

\textit{Firstly}, Ben-Ari and Pinsky ask, whether the eigenvalue of the generator of the diffusion with jump boundary possessing minimal real part is always real. Previously it has been shown in \cite{LLR} that for the one-dimensional Brownian motion with jump boundary the whole spectrum is real, but that this is in general not the case in higher dimensions. 

\textit{Secondly}, Ben-Ari and Pinsky pose the question, whether the spectral gap of a diffusion with jump boundary is always bigger than the bottom of the spectrum of the generator of the diffusion killed at the boundary. It has been shown in \cite{BaP1} that this is true in the case that the answer to the first question is affirmative, i.e. if the eigenvalue of the generator of the diffusion with jump boundary possessing minimal real part is real, then the spectral gap of the diffusion with jump boundary is always bigger than the bottom of the spectrum of the generator of the diffusion killed at the boundary.

In this article we answer these two questions in the negative. For this purpose it suffices to show that the answer to the second question is negative. In order to do so we look at the case of a one-dimensional Brownian motion in $(a,b)$ with constant drift and deterministic jump distribution. For large enough drift we are able to identify the spectral gap of the diffusion with jump boundary and deterministic jump distribution $\delta_{(a+b)/2}$ and thus can compare it with the bottom of the spectrum of the diffusion killed when exiting the interval $(a,b)$. It turns out, that the spectral gap is independent of the drift, once the absolute value of the drift is bigger than a certain threshold.

Though the main motivation behind this work can be formulated in a purely analytic way we prefer to prove our results via probabilistic methods, as the main result as well as the strategy of proof was fully inspired by the underlying probabilistic picture. Moreover our probabilistic approach explains some special features of our main result.
\section{Notation and Results}
In order to formulate our findings in a precise way and to explain the relation to the Questions of Ben-Ari and Pinsky we introduce the process in higher dimensions. Later we specialize to the one-dimensional situation. Let $D \subset \mathbb{R}^d$ be a domain with $C^{2,\alpha}$-boundary ($1>\alpha>0$) and let $(\Omega, (\mathcal{F}_t)_{t \geq 0}, (W_t)_{t \geq 0},(\mathbb{P}_x)_{x \in D})$ denote a smooth uniformly elliptic diffusion in $D$ which is killed after hitting the boundary of $D$, i.e. $(W_t)_{t \geq 0}$ is a diffusion process associated to a generator $L$ of the form 
\begin{displaymath}
L:= \frac{1}{2}\sum_{i,j=1}^da_{ij}(x)\partial_{ij} + \sum_{i=1}^d b_i(x)\partial_i,
\end{displaymath}
where the matrix $a=(a_{ij})_{i,j=1}^d$ is uniformly elliptic with symmetric coefficients $a_{ij}$, i.e. $a_{ij}=a_{ji}$. Moreover, we assume that
$a_{ij}$ and $b_i$ are bounded and have bounded derivatives. 

This process induces a compact semigroup of bounded operators $(P^D_t)_{t \geq 0}$ in $L^{2}(D)$, which is generated by $L$, with Dirichlet boundary conditions as operator acting in $L^2(D)$. This semigroup corresponds to a $L$-diffusion process in $D$, which is killed when hitting the boundary $\partial D$. The spectrum $\Sigma(-L)$ of the operator $-L$ consists of a sequence $(\lambda_k^D)_{k=0}^{\infty}$ converging to infinity. We denote by $g^D(\cdot, \cdot)$, $p^D(t,\cdot,\cdot)$ the Green- and the transition-function associated to the $L$-diffusion in $D$ killed at the boundary $\partial D$. 

Let $W^{\rho,0}$ be a $L$-diffusion in $D$  with initial distribution $\rho$ which is killed at $\partial D$. Moreover, let $(W^{\nu,i})_{i \in \mathbb{N}}$ denote an independent family of killed $L$-diffusion in $D$ with initial distribution $\nu$, which is independent of $W^{\rho,0}$. Set $T^{\rho,\nu}_0=\inf \bigl\lbrace t \geq 0 \mid W^{\rho,0}_t \in \partial D \bigr\rbrace$, $S^{\nu}_i = \inf \bigl \lbrace t \geq 0 \mid W^{\nu,i}_t \in \partial D \bigr\rbrace $ and inductively $T^{\rho,\nu}_{i+1}:= T^{\rho,\nu}_i+S^{\nu}_{i+1}$.
The process $(X^{\rho,\nu}_t)$, called a diffusion with jump boundary starting from the initial distribution $\rho$, is now defined as
\begin{equation}\label{definitionBMJB}
 X^{\rho,\nu}_t := \mathbf{1}_{\lbrace 0 \leq t < T^{\rho,\nu}_0 \rbrace}W^{\rho,0}_t+\sum_{i=1}^{\infty}\mathbf{1}_{\lbrace T_{i-1}^{\rho,\nu}\leq t < T_{i}^{\rho,\nu}\rbrace}W^{\nu,i}_{t-T^{\rho,\nu}_{i-1}}.
\end{equation}
Then it is shown in \cite{BaP1}, \cite{GK3} and \cite{KW} via different methods that there exists an invariant distribution $\pi^{L,\nu}$ and that 
\begin{equation}
\gamma_1(L,\nu) = -\lim_{t\rightarrow \infty}\frac{1}{t}\log \bigl\| \mathbb{P}_x\bigl(X_t \in \cdot\bigr)-\pi^{L,\nu}\bigr\|_{TV} > 0.
\end{equation}

The generator (in a certain sense made precise in Theorem 1 in \cite{BaP2}) of $(X_t)_{t \geq 0}$ is given by $(\mathcal{L},\mathcal{D}(\mathcal{L}))$, where 
\begin{equation}
\mathcal{D}(\mathcal{L})=\lbrace f \in C^2(\overline{D}) \mid \int_D f \,d\nu = f \restriction \partial D \rbrace, \,\text{   }\mathcal{L}f=Lf \text{  for $f\in   \mathcal{D}(\mathcal{L})$}.
\end{equation}
It is shown in \cite{BaP1} that the spectrum of $(-\mathcal{L},\mathcal{D}(\mathcal{L}))$ consists of an infinite sequence of eigenvalues which have -- with the only exception of $0$ -- strictly positive real part and that 
\begin{equation}\label{e:spectralgap}
\gamma_1(L,\nu) = \min \bigl\lbrace \Re \tilde{\lambda}\mid 0 \ne -\tilde{\lambda} \text{  is an eigenvalue for $(\mathcal{L},\mathcal{D}(\mathcal{L}))$}\bigr\rbrace
\end{equation}
As mentioned in the introduction it is known that in the case of a one-dimensional Brownian motion in $D=(a,b)$ and arbitrary jump distribution $\nu$ the spectral gap $\gamma_1(L,\nu)$ always coincides with $\lambda_1^D$. Furthermore, the full spectrum (i.e. all eigenvalues) of the generator of the one-dimensional Brownian motion with jump boundary in $D=(a,b)$ is real. On the other hand it has been shown in \cite{LLR} that in higher dimensions not necessarily every eigenvalue is real. In their examples the non-real eigenvalues do not have minimal real part.

The work \cite{BaP1} contains four questions. Question 4 concerning the continuous dependence of the spectral gap on the jump distribution is answered in \cite{KW}. Question 1 asks, whether the eigenvalue $\overline{\lambda}$ with minimal real part as in \eqref{e:spectralgap} is always real; Question 2 asks, whether $\gamma_1(L,\nu)$ is always bigger than $\lambda_0^D$. It is shown in \cite{BaP1} that an positive answer to Question 1 necessarily implies an affirmative answer to Question 2. 

Our main theorem is the following
\begin{theorem}\label{t:main}
Let $L^{\sigma,\mu}= \frac{\sigma^2}{2}\frac{d^2}{dx^2}+\mu\frac{d}{dx}$ denote the generator of a Brownian motion with variance $\sigma$ and drift $\mu$ and let $(X)_{t \geq 0}$ denote the corresponding $L^{\sigma,\mu}$-diffusion process in the interval $(a,b)$ with jump boundary and the deterministic jump distribution $\delta_{x_0}$ with $x_0=\frac{a+b}{2}$. Let $\gamma_1(L^{\sigma,\mu},\delta_{x_0})$ denote the spectral gap of $(X_t)_{t\geq 0}$. Then there exists a constant $\mu(\sigma,x_0)$ such that for all $\mu > \mu(\sigma,x_0)$
\begin{equation}\label{e:spectralgapformula}
\gamma_1(L^{\sigma,\mu},\delta_{x_0}) = \frac{8\sigma^2\pi^2}{(b-a)^2}.
\end{equation}
\end{theorem}
The assertion of Theorem \ref{t:main} shows the following somewhat surprising feature, which is connected to the fact, that $x_0$ has been defined as the center of the interval $(a,b)$. If the drift is larger than a certain threshold, then the spectral gap of the diffusion with jump distribution $\delta_{x_0}$ is independent of the drift. We remark that an explicit (not optimal) expression for $\mu(\sigma,x_0)$ can be extracted from the proof. 

Let us define
\begin{displaymath}
\mu_0(\sigma,x_0)= \inf\lbrace \mu \mid \text{ formula \eqref{e:spectralgapformula} is true}\rbrace.
\end{displaymath}
The proof of Theorem \ref{t:main} lead us to \textit{conjecture} that 
\begin{equation}\label{e:threshold}
\mu_0(\sigma,x_0) = \sqrt{3}\frac{2\sigma^2\pi}{b-a}.
\end{equation}
holds true. Thus a complementing investigation studying the case of a small drift might uncover further interesting features of this process.

The proof of Theorem \ref{t:main} actually demonstrates somewhat more, namely:
\begin{Corollary}
Let $x_0=\frac{a+b}{2}$, $L^{\sigma,\mu}=\frac{\sigma^2}{2}\frac{d^2}{dx^2} + \mu\frac{d}{dx}$ and let $(X_t)_{t \geq 0}$ be a $L^{\sigma,\mu}$-diffusion in the interval $(a,b)$ with jump boundary and jump distribution $\delta_{x_0}$. If $\mu \geq \mu(\sigma,x_0)$ then there exists an efficient coupling.
\end{Corollary}

Theorem \ref{t:main} leads immediately to the following Corollaries, answering two questions from \cite{BaP2} which we already discussed above. 
\begin{Corollary}\label{c:corolA}
Let $x_0 = (a+b)/2$ and let $L^{\sigma,\mu}= \frac{\sigma^2}{2}\frac{d^2}{dx^2}+\mu\frac{d}{dx}$ denote the generator of a Brownian motion with variance $\sigma$ and drift $\mu$ and let $(X)_{t \geq 0}$ denote the corresponding $L^{\sigma,\mu}$-diffusion process in the interval $(a,b)$ with jump boundary and the deterministic jump distribution $\delta_{x_0}$. Then for $\mu$ large enough one has
\begin{displaymath}
 \gamma_1(L^{\sigma,\mu},\delta_{x_0})< \lambda_0,
\end{displaymath}
where $\lambda_0=\frac{\sigma^2\pi^2}{2(b-a)^2}+\frac{\mu^2}{2\sigma^2}$ is the lowest eigenvalue of the selfadjoint operator $-L^{\sigma,\mu}$ in $(a,b)$ with Dirichlet boundary conditions.
\end{Corollary}
and
\begin{Corollary}\label{c:corolB}
Let $x_0 = (a+b)/2$ and let  $L^{\sigma,\mu}= \frac{\sigma^2}{2}\frac{d^2}{dx^2}+\mu\frac{d}{dx}$ denote the generator of a Brownian motion with variance $\sigma$ and drift $\mu$. Let $(X)_{t \geq 0}$ denote the corresponding $L^{\sigma,\mu}$-diffusion process in the interval $(a,b)$ with jump boundary and the deterministic jump distribution $\delta_{x_0}$ and let $\gamma_1(L^{\sigma,\mu},\delta_{x_0})$ denote the spectral gap of $(X_t)_{t \geq 0}$. Then there exists $\mu_0$ such that for $\mu \geq \mu_0$ 
\begin{displaymath}
\bigl \lbrace \tilde{\lambda} \mid  \Re \tilde{\lambda} = \gamma_1(L^{\sigma,\mu},\delta_{x_0}), \text{   $\tilde{\lambda}$ is an eigenvalue for $(-\mathcal{L},\mathcal{D}(\mathcal{L}))$}\bigr\rbrace \subset \mathbb{C}\setminus \mathbb{R}.
\end{displaymath}
\end{Corollary}
Corollary \ref{c:corolA} and Corollary \ref{c:corolB} demonstrate that the answer to Question 1 and Question 2 of \cite{BaP2} (observe the different sign convention in \cite{BaP2}) is negative.
\section{Proof of Theorem \ref{t:main}}
In this section we present the proof of our main result Theorem \ref{t:main}. In this section we assume without loss of generality that $\mu > 0$. 

In order to get a first insight into the nature of the process $(X_t)_{t \geq 0}$ for large drift $\mu$ we look at its invariant distribution $\pi_{\nu,\mu,\sigma}$.
\begin{proposition}\label{p:invmeas}
Let $L^{\sigma,\mu}:= \frac{\sigma^2}{2}\frac{d^2}{dx^2}+\mu\frac{d}{dx}$ denote the generator of a Brownian motion with variance $\sigma$ and drift $\mu$ and let $(X)_{t \geq 0}$ denote the corresponding $L^{\sigma,\mu}$-diffusion process in the interval $(a,b)$ with jump boundary and the jump distribution $\nu$. Let $\gamma_1(L^{\sigma,\mu},\delta_{x_0})$ denote the spectral gap of $(X_t)_{t\geq 0}$. Then the invariant distribution $\pi_{\nu,\mu,\sigma}$ of $(X_t)_{t\geq 0}$ satisfies
\begin{displaymath}
\lim_{\mu \rightarrow \infty}\pi_{\nu,\mu,\sigma} = \frac{ \nu((a,y])\,dy}{\int_a^b\nu((a,z])\,dz}
\end{displaymath}
in the weak sense. In particular, for $\nu = \delta_{x_0}$ we have
\begin{displaymath}
\lim_{\mu \rightarrow \infty}\pi_{\delta_{x_0},\mu,\sigma} = (b-x_0)^{-1}|\cdot\cap[x_0,b)|.
\end{displaymath}
\end{proposition}
\begin{proof}
We use the formula 
\begin{equation}\label{e:invmeas}
\pi_{\nu,\mu,\sigma}(A)=\frac{\int_A\,dy\int_a^b\nu(dx)g^{(a,b)}_{\sigma,\mu}(x,y)}{\int_a^b\,dy\int_a^b\nu(dx)g^{(a,b)}_{\sigma,\mu}(x,y)},
\end{equation}
where $g^{(a,b)}_{\sigma,\mu}(x,y)$ denotes the Green's function of the operator $L^{\sigma,\mu}$ in $(a,b)$ with Dirichlet boundary conditions. Different versions of proofs of this formula can be found e.g. in \cite{BaP1} and \cite{KW}. According to general theory (see e.g. \cite{W} section 13.2) we have the following well-known explicit formula 
\begin{equation}\label{e:green}
g^{(a,b)}_{\sigma,\mu}(x,y)=
\begin{cases}
 \frac{2\mu}{\sigma^2\bigl(e^{-2\frac{\mu}{\sigma^2}b}-e^{-2\frac{\mu}{\sigma^2}a}\bigr)}u_a(x)u_b(y)v(y) & \text{if $a< x< y < b$}\\
 \frac{2\mu}{\sigma^2\bigl(e^{-2\frac{\mu}{\sigma^2}b}-e^{-2\frac{\mu}{\sigma^2}a}\bigr)}u_b(x)u_a(y)v(y) & \text{if $a < y < x < b$},
\end{cases}
\end{equation} 
where
\begin{displaymath}
u_a(x)=\frac{\sigma^2}{2\mu}(e^{-2\frac{\mu}{\sigma^2}a}- e^{-2\frac{\mu}{\sigma^2}x})\,\text{   and   }\,u_b(x)=\frac{\sigma^2}{2\mu}\bigl(e^{-2\frac{\mu}{\sigma^2}b}-e^{-2\frac{\mu}{\sigma^2}x}\bigr)
\end{displaymath}
and
\begin{displaymath}
v(x):=e^{2\frac{\mu}{\sigma^2}x}.
\end{displaymath}
The function $u_a$ ($u_b$) is a solution to $L^{\sigma, \mu}u=0$ with $u_a(0)=0$ and $u_a'(0)=1$ ($u_b(0)=0$ and $u_b'(0)=1$). An elementary calculation gives
\begin{equation*}
g^{(a,b)}_{\sigma,\mu}(x,y)=
\begin{cases}
 \frac{\sigma^2}{2\mu}\cdot\frac{1-e^{-2\frac{\mu}{\sigma^2}(x-a)}}{e^{-2\frac{\mu}{\sigma^2}(b-a)}-1}\bigl(e^{-2\frac{\mu}{\sigma^2}(b-y)}-1\bigr) & \text{if $a< x< y < b$}\\
  \frac{\sigma^2}{2\mu}\cdot\frac{e^{-2\frac{\mu}{\sigma^2}(b-x)}-1}{e^{-2\frac{\mu}{\sigma^2}(b-a)}-1}\cdot e^{-2\frac{\mu}{\sigma^2}(x-a)}\bigl(e^{-2\frac{\mu}{\sigma^2}(b-y)}-1\bigr) & \text{if $a < y < x < b$}.
\end{cases}
\end{equation*} 
This shows that 
\begin{equation}\label{e:greenasympt}
\lim_{\mu \rightarrow \infty}\frac{2\mu}{\sigma^2}\cdot g^{(a,b)}_{\sigma,\mu}(x,y)=
\begin{cases}
1 &\text{if $a< x< y < b$}\\
0 &\text{if $a < y < x < b$}
\end{cases}
\end{equation}
Formulas \eqref{e:invmeas} and \eqref{e:greenasympt} together with an elementary application of Lebesgue's theorem of dominated convergence now imply the assertion of the proposition.
\end{proof}
In particular in the case of a deterministic jump distribution Proposition \ref{p:invmeas} is intuitively plausible. If the jump distribution is $\delta_{x_0}$ and the drift gets large one effectively ends up with a diffusion on the reduced interval $(x_0,b)$, as the strong drift does not allow to visit the region $(a,x_0)$. In the large drift limit one thus can expect that the diffusion essentially looks like a motion on the circle with diameter $b-x_0$. This explains the form of the large drift limit of the invariant distribution. On the other hand this probabilistic heuristic reasoning leads to the conjecture that the spectral gap does remain bounded in the large drift limit, whereas the bottom of the spectrum of the operator $L^{\sigma,\mu}$ in $(a,b)$ with absorbing boundary of course tends to infinity. The bottom of the spectrum is calculated in the following
\begin{lemma}\label{l:exittime}
Let $(X_t)_{t \geq 0}=(\sigma B_t + \mu t)_{t \geq 0}$ denote a one-dimensional Brownian motion with variance $\sigma$ and drift $\mu$ and let $\tau_{(a,b)}$ be the first exit time from the interval $(a,b)$. Then 
\begin{equation}\label{e:exit}
-\lim_{t \rightarrow \infty}\frac{1}{t}\log \sup_{x \in (a,b)}\mathbb{P}_x\bigl(\tau_{(a,b)} > t\bigr) = \frac{\sigma^2 \pi^2}{2(b-a)^2} + \frac{\mu^2}{2\sigma^2}.
\end{equation}
In particular the bottom of the spectrum of $-L^{\sigma,\mu}=-\frac{\sigma^2}{2}\frac{d^2}{dx^2}-\mu\frac{d}{dx}$ in $(a,b)$ with Dirichlet boundary conditions coincides with the right hand side of \eqref{e:exit}.
\end{lemma}
\begin{proof}
The process $(X_t)_{t \geq 0}$ is reversible and it is generated by the closure $q$ of the symmetric Dirichlet form
\begin{displaymath}
C_c^{\infty}((a,b))\ni f \mapsto  \frac{\sigma^2}{2}\int_a^b|f'(x)|^2\exp{(2\frac{\mu}{\sigma^2}x)}\,dx,
\end{displaymath}
considered in the Hilbert space $L^2((a,b),e^{2\frac{\mu}{\sigma^2}x}\,dx)$. It is well-known (see section 3.6 in \cite{P95}) that the left hand side of equation \eqref{e:exit} coincides with the bottom of the spectrum of the operator $L= -\frac{\sigma^2}{2}\frac{d^2}{dx^2} - \mu\frac{d}{dx}$ (with Dirichlet boundary conditions), which is uniquely associated to the Dirichlet form $q$. Set 
\begin{displaymath}
U:L^2((a,b),dx)\rightarrow L^2((a,b),e^{2\frac{\mu}{\sigma^2}x}\,dx),\, Uf(x) = e^{-\frac{\mu}{\sigma^2}x}f(x).
\end{displaymath}
Then $U$ is obviously unitary and a direct calculation shows that $U^{-1}LU$ coincides with the operator $-\frac{\sigma^2}{2} \frac{d^2}{dx^2} + \frac{\mu^2}{2\sigma^2}$, considered as a selfadjoint operator in $L^2((a,b),dx)$ with Dirichlet boundary condition. Since $\frac{\sigma^2\pi^2}{2(b-a)^2}$ is the smallest eigenvalue of $-\frac{\sigma^2}{2} \frac{d^2}{dx^2}$ in the interval $(a,b)$ with Dirichlet boundary conditions, we conclude that
\begin{displaymath}
-\lim_{t \rightarrow \infty}\frac{1}{t}\log \sup_{x \in (a,b)}\mathbb{P}_x\bigl(\tau_{(a,b)} > t\bigr) = \frac{\sigma^2 \pi^2}{2(b-a)^2} + \frac{\mu^2}{2\sigma^2}.
\end{displaymath} 
\end{proof}
An alternative proof of Lemma \ref{l:exittime} might be based on the Girsanov formula, of course. As expected the bottom of the spectrum of $-L^{\sigma,\mu}$ diverges to infinity if $|\mu| \rightarrow \infty$.\\
The following relations for the total-variation-distance are well-known and will be used throughout this paper.
It can be readily established that
\begin{equation}
d(t)=\sup_{x\in (a,b)}\bigl\|\mathbb{P}_x\bigl(X_t\in\cdot)-\pi_{\mu,\nu,\sigma}(\cdot) \bigr\|_{TV}\le \sup_{x,y\in (a,b)}\bigl\|\mathbb{P}_x\bigl(X_t\in\cdot\bigr)-\mathbb{P}_y\bigl(X_t\in\cdot)\bigr\|_{TV}=\bar{d}(t)
\end{equation}
On the other hand, by triangular- and coupling-inequality we have
\begin{equation}\label{didibarinequality}
\bar{d}(t)\le 2d(t)\le 2\mathbb{P}\bigl(\tau_{coup}>t\bigr), 
\end{equation}
where $\tau_{coup}$ denotes the coupling time of both processes defined below. Hence, we aim to find bounds for $\mathbb{P}\bigl(\tau_{coup}>t\bigr)$. 
>From \eqref{didibarinequality} it follows that 
\[
\gamma_1(L^{\sigma,\mu},\delta_{x_0})\ge - \lim_{t\rightarrow\infty}\frac{1}{t}\log \mathbb{P}\bigl(\tau_{coup}>t\bigr).
\]

\subsection{Upper Bound on $\gamma_1(L^{\sigma,\mu},\delta_{x_0})$}
In this subsection we establish an upper bound on $\gamma_1(L^{\sigma,\mu},\delta_{x_0})$. We start with several Lemmata, which will be used in the proof of the upper bound for the spectral gap.
\begin{lemma}\label{l:epsilonterm}
For $c<0<d$ let
\begin{displaymath}
\tau_{(c,d)}=\inf\lbrace t >0\mid |B_t| \not\in (c,d) \rbrace
\end{displaymath}
denote the first exit time of a standard Brownian motion $(B_t)_{t \geq 0}$ from the interval $I:=(c,d)$. Then there exists $\varepsilon >0$ and $\tilde{t}>0$ such that for all $t\geq \tilde{t}$
\begin{displaymath}
\mathbb{P}(\tau_{(c,d)} \in [t-\varepsilon,t])\leq \frac{1}{2}\mathbb{P}(\tau_{(c,d)}>t).
\end{displaymath} 
\end{lemma}
\begin{proof}
Let $(\bar{\lambda}^I_k)_{k=0}^{\infty}$ be the sequence of eigenvalues of the operator $-\frac{1}{2}\frac{d^2}{dx^2}$ in $I=(c,d)$ with Dirichlet boundary conditions and let $(\varphi_k^I)_{k \in \mathbb{N}}$ denote the sequence of associated eigenfunctions. By the spectral decomposition of Brownian motion, killed when exiting the interval $I$, we have
\begin{equation}
\begin{split}
\mathbb{P}(\tau_{(c,d)}>t) &= \sum_{i=0}^{\infty}e^{-\bar{\lambda}_i^It}\varphi_i^I(0)\int_{I}\varphi_i^I(y)\,dy \\
&= e^{-\bar{\lambda}_0^It}\varphi_0^I(0)\int_I\varphi_0^I(y)\,dy + O(e^{-\bar{\lambda}_1^It}).
\end{split}
\end{equation}
Hence we have for some suitable constant $C>0$
\begin{equation}
\begin{split}
\mathbb{P}(\tau_{(c,d)} \in [t-\varepsilon,t]) &= \sum_{i=0}^{\infty}[e^{-\bar{\lambda}_i^I(t-\varepsilon)}-e^{-\bar{\lambda}_i^It}]\varphi_i^I(0)\int_I\varphi_i^I(y)\,dy \\
&= e^{-\bar{\lambda}_0^It}(e^{\bar{\lambda}_0^I\varepsilon}-1)\varphi_0^I(0)\int_I\varphi_0^I(y)\,dy + Ce^{-\bar{\lambda}_1^I(t-\varepsilon)}.
\end{split}
\end{equation}
We can obviously choose $\varepsilon$ small enough such that 
\begin{displaymath}
(e^{\bar{\lambda}_0^I\varepsilon}-1)\varphi_0^I(0)\int_I\varphi_0^I(y)\,dy \leq \frac{1}{4}.
\end{displaymath}
Moreover, there exists $\tilde{t}$ such that for every $t>\tilde{t}$ one has $Ce^{-\bar{\lambda}_1^I(t-\varepsilon)} \leq \frac{1}{4}e^{-\bar{\lambda}_0^It}$. This proves the assertion of the Lemma.
\end{proof}
In the following Lemma we show, that for certain pairs of deterministic initial distribution there exists a coupling such that both diffusions with jump boundary coalesce quite fast.
\begin{lemma}\label{l:fastcoupl}
We denote by $x_0=(a+b)/2$ the center of the interval $(a,b)$ and for $a < x < x_0$ we set $x_S:x+(b-a)/2$. Then for all $t> 0$
\begin{equation*}
\begin{split}
 \sup_{x \in (a,x_0)}\bigl\| \mathbb{P}_x(X_t \in \cdot) - \mathbb{P}_{x_S}(X_t \in \cdot)\bigr\|_{TV} & \leq \sup_{x \in (a,x_0)}\mathbb{P}_{x}(\tau > t)\\
 &\leq e^{\frac{b-a}{2}\frac{\mu}{\sigma^2}}e^{-\Lambda(\mu) t},
\end{split}
\end{equation*}
where $\tau$ is the first exit time of $(\sigma B_t + \mu t)_{t \geq 0}$ from the interval $(a,x_0)$ and $\Lambda(\mu):=\frac{\mu^2}{2\sigma^2}$.
\end{lemma}  
\begin{proof}
Let $\lambda > 0$ be given. Let $L^{\sigma,\mu}:=\frac{\sigma^2}{2}\frac{d^2}{dx^2}+\mu\frac{d}{dx}$. Set 
\begin{displaymath}
X_t=x + \mu t + \sigma B_t\,\text{   and   }\,Y_t=x+ \frac{b-a}{2} + \mu t + \sigma B_t = x_S+ \mu t + \sigma B_t
\end{displaymath} 
and continue the process such that both give $L^{\sigma,\mu}$-diffusions with jump boundary and jump distribution $\delta_{x_0}$. Then immediately after the random time
\begin{displaymath}
\tau = \inf\bigl\lbrace t>0 \mid X_t \notin (a,b)\, \text{  or  }\,Y_t \notin (a,b)\bigr\rbrace
\end{displaymath}
the processes $(X_t)_{t \geq 0}$ and $(Y_t)_{t \geq 0}$ coalesce. Therefore according to the coupling inequality it follows that for all $a \leq x \leq x_0$
\begin{displaymath}
\bigl\| \mathbb{P}_x(X_t \in \cdot) - \mathbb{P}_{x_S}(X_t \in \cdot)\bigr\|_{TV} \leq \sup_{x \in (a,x_0)}\mathbb{P}_x(\tau > t).
\end{displaymath}
The random time $\tau$ can be rewritten as a first exit time of a Brownian motion with drift, more precisely
\begin{displaymath}
\tau=\inf\bigl\lbrace t > 0 \mid x + \mu t + \sigma B_t \notin \bigl(a, (a+b)/2 \bigr) \bigr \rbrace.
\end{displaymath}
This gives the first inequality. Using the unitary equivalence $U$ from the proof of Lemma \ref{l:exittime} with $L$ (or alternatively Girsanov's formula)  we have
\begin{equation*}
\begin{split}
\mathbb{P}_x(\tau > t) &= e^{tL^{\sigma,\mu}}\mathbf{1}(x) = \bigl[Ue^{t\tilde{L}}(U^{-1}\mathbf{1})\bigr](x) \\
&= e^{-\frac{\mu}{\sigma^2} x}e^{-\frac{\mu^2}{2\sigma^2}t}\mathbb{E}_x\bigl[e^{\frac{\mu}{\sigma^2} B_t};\forall s\leq t:B_s \in (a,(a+b)/2)\bigr] \\
&\leq   e^{-\frac{\mu}{\sigma^2} a}e^{\frac{\mu}{\sigma^2}(a+b)/2} e^{-\frac{\mu^2}{2\sigma^2}t}= e^{\frac{\mu}{\sigma^2}\frac{b-a}{2}}e^{-\frac{\mu^2}{2\sigma^2}t},
\end{split}
\end{equation*}
where $\tilde{L}:=\frac{1}{2}\frac{d^2}{dx^2}-\frac{\mu^2}{2\sigma^2}$ is the selfadjoint operator in $L^2(a,b)$ with Dirichlet boundary conditions. This implies the assertion of the Lemma.
\end{proof}
\begin{lemma}\label{l:exponentialconvol}
Let $c<0<d$ and let $\tau_{(c,d)}$ denote the first exit time of a standard Brownian motion from the interval $(c,d)$ and let $\tau$ denote the first exit time of $(\sigma B_t + \mu t)_{t \geq 0}$ from the interval $(a,(a+b)/2)$.  Then there exist constants $\tilde{\mu}>0$ and $0<c<1$ such that for $\mu > \tilde{\mu}$ and $t\geq \tilde{t}$
\begin{displaymath}
\int_0^t \sup_{x \in (a,x_0)}\mathbb{P}_x(\tau > t-s)\mathbb{P}(\tau_{(c,d)} \in ds) \leq c\mathbb{P}(\tau_{(c,d)} > t).
\end{displaymath}
\end{lemma}
\begin{proof}
First we obviously have
\begin{equation}\label{e:decomp}
\begin{split}
\int_0^t  \sup_{x \in (a,x_0)}\mathbb{P}_x(\tau > t-s) \mathbb{P}(\tau_{(c,d)} \in ds) &= \int_0^{t-\varepsilon} \sup_{x \in (a,x_0)}\mathbb{P}_x(\tau > t-s)\mathbb{P}(\tau_{(c,d)} \in ds) \\
&+ \int_{t-\varepsilon}^t \sup_{x \in (a,x_0)}\mathbb{P}_x(\tau > t-s)\mathbb{P}(\tau_{(c,d)} \in ds)
\end{split}
\end{equation}
for every $\varepsilon > 0$. According to Lemma \ref{l:epsilonterm} we can choose $\varepsilon > 0$ such that for some $\tilde{t} \in (0,\infty)$ and all $t\geq \tilde{t}$ 
\begin{equation}
\int_{t-\varepsilon}^t \sup_{x \in (a,x_0)}\mathbb{P}_x(\tau > t-s) \mathbb{P}(\tau_{(c,d)} \in ds) \leq \mathbb{P}(\tau_{(c,d)} \in [t-\varepsilon,t]) \leq \frac{1}{2}\mathbb{P}(\tau_{(c,d)}>t).
\end{equation}
For the first term in \eqref{e:decomp} observe first that according to Lemma \ref{l:fastcoupl} and the integration by parts formula

\begin{equation}\label{e:intparts}
\begin{split}
\int_0^{t-\varepsilon}  &\sup_{x \in (a,x_0)}\mathbb{P}_x(\tau > t-s) \mathbb{P}(\tau_{(c,d)} \in ds)\le e^{\frac{b-a}{2}\frac{\mu}{\sigma^2}} \int_0^{t-\varepsilon} e^{-\Lambda(\mu)(t-s)}\,\mathbb{P}\bigl(\tau_{(c,d)}\in ds\bigr)\\ 
&=-e^{\frac{b-a}{2}\frac{\mu}{\sigma^2}}e^{-\Lambda(\mu)\varepsilon}\mathbb{P}(\tau_{(c,d)} > t-\varepsilon) \\
& \,+\, e^{\frac{b-a}{2}\frac{\mu}{\sigma^2}}e^{-\Lambda(\mu) t}+ e^{\frac{b-a}{2}\frac{\mu}{\sigma^2}}\Lambda(\mu) \int_0^{t-\varepsilon}e^{-\Lambda(\mu) (t-s)}\mathbb{P}(\tau_{(c,d)} > s)\,ds \\
&\le e^{\frac{b-a}{2}\frac{\mu}{\sigma^2}}e^{-\Lambda(\mu) t}+ e^{\frac{b-a}{2}\frac{\mu}{\sigma^2}}\Lambda(\mu) \int_0^{t-\varepsilon}e^{-\Lambda(\mu) (t-s)}\mathbb{P}(\tau_{(c,d)} > s)\,ds.\\
\end{split}
\end{equation}
The first term in the last line is obviously not at all a problem for $t>\tilde{t}$, while the second has to be analyzed more carefully.
If $\bar{\lambda}_0$ denotes the lowest eigenvalue of $-\frac{1}{2}\frac{d^2}{dx^2}$ in $(c,d)$ with Dirichlet boundary conditions, we obviously have
$\lim_{t \rightarrow \infty}e^{\bar{\lambda}_0t}\mathbb{P}(\tau_{(c,d)} > t) = \tilde{C} < \infty$ there exists a constant $C>\tilde{C}$ such that for all $t\geq 0$ 
\begin{displaymath}
\mathbb{P}(\tau_{(c,d)} > t) \leq Ce^{-\bar{\lambda}_0 t}.
\end{displaymath}  
Therefore we get for the second term in \eqref{e:intparts} ($t \geq \tilde{t}$)
\begin{equation}\label{e:lastterm}
\begin{split}
e^{\frac{b-a}{2}\frac{\mu}{\sigma^2}}\Lambda(\mu) \int_0^{t-\varepsilon} &e^{-\Lambda(\mu)(t-s)}\mathbb{P}(\tau_{(c,d)} \geq s)\,ds \\
&\leq  e^{\frac{b-a}{2}\frac{\mu}{\sigma^2}}\frac{C\,\Lambda(\mu)}{\Lambda(\mu) - \bar{\lambda}_0}e^{-\Lambda(\mu) t}\bigl(e^{(\Lambda(\mu)-\bar{\lambda}_0) (t-\varepsilon)}-1\bigr) \\
&\leq e^{\frac{b-a}{2}\frac{\mu}{\sigma^2}+\bar{\lambda}_0\varepsilon-\Lambda(\mu)\varepsilon}\frac{C\,\Lambda(\mu)}{\Lambda(\mu) - \bar{\lambda}_0}e^{-\bar{\lambda}_0t}.
\end{split}
\end{equation}
Using that $\Lambda(\mu)$ is quadratic in $\mu$ we see that the last term in \eqref{e:lastterm} is of the form $c_2(\mu)e^{-\bar{\lambda}_0t}$ with $\lim_{\mu \rightarrow \infty}c_2(\mu)=0$. Thus we have shown that for $t \geq \tilde{t}$
\begin{equation*}
\int_0^t\sup_{x \in (a,x_0)} \mathbb{P}_x(\tau > t-s)\mathbb{P}(\tau_{(c,d)} \in ds) \leq c(\mu)e^{-\bar{\lambda}_0t} + \frac{1}{2}\mathbb{P}(\tau_{(c,d)} > t).
\end{equation*} 
where $c$ satisfies $\lim_{\mu \rightarrow \infty}c(\mu)=0$. Taking $\tilde{\mu}$ large enough we conclude that for $t \geq \tilde{t}$ we have $c(\mu)e^{-\bar{\lambda}_0t} \leq \frac{1}{4}\mathbb{P}(\tau_{(c,d)} > t)$, which proves the assertion of the Lemma.
\end{proof}

In order to get an upper bound on $\gamma_1(L^{\mu,\sigma},\delta_{x_0})$ it follows from \eqref{didibarinequality} that it suffices to look for lower bounds on 
\begin{equation}
\bar{d}^{\sigma,\mu}(t):=\sup_{x,y \in (a,b)}\bigl\|\mathbb{P}_x(X_t\in \cdot)-\mathbb{P}_y(X_t \in \cdot)\bigr\|_{TV},
\end{equation}
where the process $(X_t)_{t\geq 0}$ is a Brownian motion in $(a,b)$ with variance $\sigma$ and drift $\mu$ and jump distribution $\delta_{x_0}$ as also specified in Theorem \ref{t:main}. Recall that we assume without loss of generality that $\mu > 0$.

Let us define several quantities, which we use in the sequel:
\begin{equation*}
\begin{split}
\bullet\,\, t_n:=\frac{b-x_0}{\mu}n,\,\text{   }\,x=x_1:= x_0 + \frac{b-x_0}{4},\,\text{   }x_2:=x_0 + \frac{b-x_0}{2}\,\text{  and  }\,y:=x_3= x_0+\frac{3(b-x_0)}{4}.
\end{split}
\end{equation*}
Moreover, we set 
\begin{equation*}
\bullet\,\,A:=[x_0,x_2)\text{  and  }\,J:=\bigl(-\frac{b-x_0}{4\sigma}, \frac{b-x_0}{4\sigma}\bigr). 
\end{equation*}
The next lemma explains to some extent these choices. In this lemma we use the following representation of the Brownian motion in $(a,b)$ with constant variance $\sigma$, constant drift $\mu$ and jump distribution $\delta_{x_0}$. Let $(B_t)_{t \geq 0}$ be a standard one-dimensional Brownian motion, then we can build a path of the Brownian motion with constant drift $\mu$ and jump distribution $\delta_{x_0}$ and start in $x$:
\begin{equation}\label{e:represI}
X_t:=
( x + \sigma B_t + \mu t)\mathbf{1}_{\lbrace t < T_0 \rbrace }+ \sum_{i=1}^{\infty}[ x_0 + \sigma(B_{t}- B_{T_{i-1}}) + \mu (t-T_{i-1})]\mathbf{1}_{\lbrace T_{i-1} \leq t < T_i \rbrace }
\end{equation}
where 
\begin{displaymath}
\begin{split}
&T_0 := \inf\bigl\lbrace t > 0\mid x + \sigma B_t + \mu t \notin (a,b)\bigr\rbrace \text{   and   } \\
&T_{i} := \inf\bigl\lbrace t > T_{i-1} \mid x_0 + \sigma(B_{t}- B_{T_{i-1}}) + \mu (t-T_{i-1})\notin (a,b) \bigr\rbrace
\end{split}
\end{displaymath}
Let 
\begin{equation}\label{e:represII}
\tau_J:= \inf \bigl\lbrace t > 0 \mid B_t \notin \bigl(-\frac{b-x_0}{4 \sigma},\frac{b-x_0}{4\sigma}\bigr)\bigr\rbrace.
\end{equation}
The following Lemma is very elementary but we include a rather detailed proof, since the assertion might be confusing at the first sight and since it allows us to introduce the paths $\xi_l$ and $\xi_r$.
\begin{lemma}\label{l:specialchoice}
Let $(t_n)_{n \in \mathbb{N}}$, $x_0,x,x_2,y$ and $A$ be chosen as above. Then using the notation introduced in \eqref{e:represI} and \eqref{e:represII}, one has
\begin{equation*}
\mathbb{P}_x(X_{t_n} \in A,\tau_J > t_n) = \mathbb{P}_x(\tau_J > t_n)\,\text{   and   }\,\mathbb{P}_y(X_{t_n} \in A,\tau_J > t_n)=0.
\end{equation*} 
\end{lemma}
\begin{proof}
The proof proceeds via a pathwise argument. For the first assertion it suffices to show that 
\begin{displaymath}
\lbrace \tau_J > t_n \rbrace \subset  \lbrace X_{t_n} \in A,\tau_J > t_n \rbrace. 
\end{displaymath}
Let $(B_t(\omega))_{t \geq 0}$ be a realization of a Brownian path started in $0$ which exits the interval $J$ after time $t_n$, i.e. $\omega \in \lbrace \tau_J > t_n \rbrace$ and let $(X_t(\omega))_{t\geq 0}$ be a realization of our diffusion with jump boundary defined as in \eqref{e:represI}. 
First observe that due to the restriction on $B_t(\omega)$ one has $X_{T_1-}(\omega) = b$ and $T_0(\omega) < t_1$.

Let $n \in \mathbb{N}$ be given. In order to prove the first formula let us define two deterministic versions of our diffusion with jump boundary. The path $\xi_r$ is defined in the following way 
\begin{equation}\label{e:xi_rT_0}
\xi_r(t) :=
\begin{cases}
x_2 + \mu t &\text{ if $0 \leq t < T_0(\omega)$}\\
x_0 + [\xi_r(T_0(\omega)-)-b] + \mu(t-T_0(\omega)) &\text{ if $T_0(\omega) \leq t \leq t_1$}
\end{cases}
\end{equation}
In an analogous way we define 
\begin{equation}\label{e:xi_lT_0}
\xi_l(t) :=
\begin{cases}
x_0 + \mu t &\text{ if $0 \leq t < T_0(\omega)$}\\
x_0 + [\xi_l(T_0(\omega)-)-b] + \mu(t-T_0(\omega)) &\text{ if $T_0(\omega) \leq t \leq t_1$}
\end{cases}
\end{equation}
Observing that
\begin{equation}\label{e:jumptime}
\mu T_0(\omega)=b-x-\sigma B_{T_0}(\omega).
\end{equation}  
a simple calculation using \eqref{e:jumptime} shows that for $t_1 \geq t \geq T_0(\omega)$
\begin{equation}\label{e:distance}
X_t(\omega)-\xi_l(t) = \frac{b-x_0}{4} + \sigma B_t(\omega) =:\Delta^1_l(t) \text{  and  }\xi_r(t)-X_t(\omega) =  \frac{b-x_0}{4} - \sigma B_t(\omega) =:\Delta^1_r(t).
\end{equation}
Equation \eqref{e:distance} together with our restriction on the Brownian path $B_t(\omega)$ show that $T_1(\omega)>t_1$. At the time $t_1$ we have $\xi_l(t_1)= x_0$ and $\xi_r(t_1) = x_2$ and we have $\forall\, T_0(\omega) \leq t \leq t_1:\Delta_l^1(t),\Delta_r^1(t) \geq 0$. Thus $\forall\, T_0(\omega) \leq t \leq t_1:  \xi_{l}(t) \leq X_t(\omega) \leq \xi_r(t)$ and in particular $X_{t_1}(\omega) \in A$. Thus if $n = 1$ we are done. 

If $n \geq 2$, we first observe from \eqref{e:distance} that $T_1(\omega) < t_2$ and continue $\xi_l$ and $\xi_r$ for $T_0(\omega) \leq t \leq T_1(\omega)$ in the obvious way and for $T_1(\omega) \leq t \leq t_2$ by setting 
\begin{equation*}
\xi_r(t) = x_0 + (\xi_r(T_1(\omega)-)-b) + \mu(t-T_1(\omega))\, \text{   $T_1(\omega)\leq t \leq t_2$}. 
\end{equation*}
and
\begin{equation*}
\xi_l(t) =
 x_0 - (b-\xi_l(T_1(\omega)-)) + \mu(t-T_1) \,\text{   $T_1(\omega) \leq t \leq t_2$}. 
\end{equation*}
Using the same steps as in the case $n=1$ one calculates that for $T_1(\omega) \leq  t \leq t_2$
\begin{equation}\label{e:distanceII}
X_t(\omega)-\xi_l(t) = \frac{b-x_0}{4} + \sigma B_t(\omega) =:\Delta^2_l(t) \text{  and  } \xi_r(t)-X_t(\omega) =  \frac{b-x_0}{4} - \sigma B_t(\omega) =:\Delta^2_r(t)
\end{equation} 
showing that  $T_2(\omega) > t_2$ and that for $T_1(\omega) \leq t \leq t_2$ by the restriction on $B_t(\omega)$ we have $\xi_l(t)\leq X_t(\omega) \leq \xi_r(t)$. In particular $x_0=\xi_l(t_2) \leq X_{t_2}(\omega) \leq \xi_r(t_2)=x_2$. If $n =2$ we are done, if $n>2$ we can continue the process in the same manner. 

The second assertion follows via an essentially analogous argument and the proof is skipped.
\end{proof} 
Using the special choices made above and Lemma \ref{l:specialchoice} we conclude via the inverse triangle inequality that 
\begin{equation}\label{e:essentialsplitting}
\begin{split}
\bar{d}(t_n) &\geq \bigl|\mathbb{P}_x\bigl(X_{t_n} \in A)-\mathbb{P}_y\bigl(X_{t_n} \in A\bigr)\bigr| \\
&\geq \bigl|\mathbb{P}_x\bigl(X_{t_n} \in A,\tau_J>t_n)-\mathbb{P}_y\bigl(X_{t_n} \in A,\tau_J>t_n\bigr)\bigr| \\
&\, -\bigl|\mathbb{P}_x\bigl(X_{t_n} \in A,\tau_J \leq t_n)-\mathbb{P}_y\bigl(X_{t_n} \in A,\tau_J \leq t_n\bigr)\bigr|\\
&= \mathbb{P}\bigl(\tau_J > t_n\bigr)-\bigl|\mathbb{P}_x\bigl(X_{t_n} \in A,\tau_J \leq t_n)-\mathbb{P}_y\bigl(X_{t_n} \in A,\tau_J \leq t_n\bigr)\bigr| 
\end{split}
\end{equation}
The essential point is that the first term in the last line of \eqref{e:essentialsplitting} is the leading term. This will be shown in the last Lemma of this subsection.

Due to
\begin{equation*}
\begin{split}
\bigl|\mathbb{P}_x\bigl(X_{t_n} \in A,\tau_J \leq t_n)&-\mathbb{P}_y\bigl(X_{t_n} \in A,\tau_J \leq t_n\bigr)\bigr| \\
&\leq \int_0^{t_n}\bigl|\mathbb{P}_x\bigl(X_{t_n}\in A\mid\tau_J=s\bigr)-\mathbb{P}_y\bigl(X_{t_n}\in A\mid\tau_J=s\bigr)\bigr|\mathbb{P}(\tau_J \in ds)
\end{split}
\end{equation*} 
we arrive at 
\begin{equation}\label{e:lowerboundI}
\begin{split}
\bar{d}(t_n)&\geq \mathbb{P}(\tau_J > t_n) - \int_0^{t_n}\bigl|\mathbb{P}_x\bigl(X_{t_n}\in A\mid\tau_J=s\bigr)-\mathbb{P}_y\bigl(X_{t_n}\in A\mid\tau_J=s\bigr)\bigr|\mathbb{P}(\tau_J \in ds)
\end{split}
\end{equation}
In order to do the last step in the proof of Theorem \ref{t:main} we first note that for the process $(X_t)_{t \geq 0}$ with $X_0 = x$
\begin{equation*}
\begin{split}
X_{\tau_J} = 
\begin{cases}
\xi_l(\tau_J) & \text{ if $B_{\tau_J} = -\frac{b-x_0}{4\sigma}$ }\\
\xi_r(\tau_J) & \text{ if $B_{\tau_J} = \frac{b-x_0}{4\sigma}$ and $\xi_r(\tau_J) < b$} \\
\xi_r(\tau_J)-(b-a)/4 & \text{ if $B_{\tau_J} = \frac{b-x_0}{4\sigma}$ and $\xi_r(\tau_J) \geq b$}
\end{cases}
\end{split}
\end{equation*}
Now, let $(\tilde{X}_t)_{t \geq 0}$ be the $L^{\sigma,\mu}$-diffusion in $(a,b)$ with jump boundary and jump distribution $\delta_{x_0}$ and $\tilde{X}_0 = y$. Comparison with the in $x$ started process $(X_t)_{t \geq 0}$ demonstrates that the following property holds true:
\begin{equation}\label{e:startiny}
\text{\textbf{Either} we have  $\tilde{X}_{\tau_J} = X_{\tau_J}$  \textbf{or} we have $|\tilde{X}_{\tau_J} - X_{\tau_J}| = \frac{b-a}{2}$.}
\end{equation}
Moreover, each of the two possibilities in \eqref{e:startiny} occurs with probability $1/2$.
\begin{lemma}\label{l:smallconstant}
There exist a constant $\tilde{\mu}(\sigma,x_0)$, $n_0 \in \mathbb{N}$ and a constant $c < 1$ such that for $\mu \geq \tilde{\mu}(\sigma,x_0)$ and $n\geq n_0$
\begin{displaymath}
\int_0^{t_n}\bigl|\mathbb{P}_x\bigl(X_{t_n}\in A\mid\tau_J=s\bigr)-\mathbb{P}_y\bigl(X_{t_n}\in A\mid\tau_J=s\bigr)\bigr|\mathbb{P}(\tau_J \in ds) < c\,\mathbb{P}(\tau_J > t_n).
\end{displaymath}
\end{lemma}
\begin{proof}
First choose $n_0 \in \mathbb{N}$ such that $t_n \geq \tilde{t}$ for every $n\geq n_0$. Now observe that due to the strong Markov property and \eqref{e:startiny} one has
\begin{equation}
\begin{split}
\bigl|\mathbb{P}_x\bigl(X_{t_n}\in A\mid\tau_J=s\bigr) &- \mathbb{P}_y\bigl(X_{t_n}\in A\mid\tau_J=s\bigr)\bigr| \\
&\leq \sup_{z \in (a,x_0)}\bigl|\mathbb{P}_z\bigl(X_{t_n-s} \in A\bigr)-\mathbb{P}_{z_S}\bigl(X_{t_n-s}\in A\bigr)\bigr|,
\end{split}
\end{equation}
where as in Lemma \ref{l:fastcoupl} $z_S=z+(b-a)/2$. Therefore we get via Lemma \ref{l:fastcoupl}
\begin{equation*}
\begin{split}
\int_0^{t_n}\bigl|\mathbb{P}_x\bigl(X_{t_n}\in A\mid\tau_J=s\bigr)&-\mathbb{P}_y\bigl(X_{t_n}\in A\mid\tau_J=s\bigr)\bigr|\mathbb{P}(\tau_J \in ds) \\
&\leq \int_0^{t_n}\sup_{z \in (x_0-(b-x_0)/4,x_0)}\bigl|\mathbb{P}_z\bigl(X_{t_n-s}\in A\bigr)-\mathbb{P}_{z_S}\bigl(X_{t_n-s}\in A\bigr)\bigr|\mathbb{P}(\tau_J \in ds) \\
&\leq \int_0^{t_n}\sup_{x \in (a,x_0)}\mathbb{P}_x(\tau > t-s)\mathbb{P}(\tau_J \in ds),
\end{split}
\end{equation*}
where as in Lemma \ref{l:fastcoupl} $\tau$ denotes the first exit time of $(\sigma B_t + \mu t)_{t \geq 0}$ from $(a,x_0)$. An application of Lemma \ref{l:exponentialconvol} immediately implies the assertion. 
\end{proof}
Lemma \ref{l:smallconstant} and inequality \eqref{e:lowerboundI} imply that for large enough $\mu$ and all $\mathbb{N} \ni n \geq n_0$ we have
\begin{equation}\label{e:lowerboundII}
\bar{d}(t_n) \geq (1-c)\mathbb{P}(\tau_J > t_n).
\end{equation}
Since the exponential tails of the exit time $\tau_J$ are easily calculated we have shown that 
\begin{equation}
\gamma_1(L^{\sigma,\mu},\delta_{x_0}) \leq \frac{8\sigma^2\pi^2}{(b-a)^2}
\end{equation}
for all $\mu\ge \mu_0$ and $\mu_0$ sufficiently large.
\subsection{Lower Bound on $\gamma_1(L^{\sigma,\mu},\delta_{x_0})$:} It will turn out once again that coupling methods are a powerful tool in getting bounds for the rate of convergence of a Markov process towards its invariant distribution. Similar to \cite{KW} we construct two suitable versions of diffusions with jump boundary and jump distribution $\delta_{x_0}$ corresponding to different initial distributions simultaneously and control the tails of the coupling time.\\
We will need the following very elementary auxiliary result:
\begin{proposition}\label{firstprop}
Let $I=(a,b)$ be an open interval with center $x_0=\frac{a+b}{2}$. For $y\in I$, let $\tau_{y}=\inf\{t:y+B_t\in\partial I\}$ the first time
of leaving the interval $I$. Then we have for all $y\in I$ and $t\in\mathbb{R}_+$
\[
\mathbb{P}\bigl(\tau_{y}>t\bigr)\le \mathbb{P}\bigl(\tau_{x_0}>t\bigr).
\] 
\end{proposition}
\begin{proof}
The proof is based on a simple coupling argument: Without loss of generality we may assume that $b>y>x_0$. We define the coupling as follows:
For $t<\tau^{x_0y}=\inf\{t:X_t=Y_t\}$ let
$X_t=y-B_t$ and $Y_t=x_0+B_t$. Now let us distinguish the following two cases:
First case: $X$ and $Y$ meet in $\frac{x_0+y}{2}$ at time $\tau^{x_0y}<\tau_{x_0}$. In this case let
$Y=X=\frac{x_0+y}{2}+B_t -B_{\tau_{x_0}}$ for $t\ge \tau^{x_0y}$ and hence both processes leave the interval at the same time.\\
Second case: $X$ and $Y$ do not meet each other before $\tau_{x_0}$. In this case, we have by definition of the processes that
$\tau_y <\tau_{x_0}$. Hence we obtain
\begin{eqnarray*}
\mathbb{P}\bigl(\tau_{x_0}>t\bigr)&=&\mathbb{P}\bigl(\tau_{x_0}>t, \tau_{x_0}>\tau^{x_0y}\bigr)+\mathbb{P}\bigl(\tau_{x_0}>t,\tau_{x_0}<\tau^{x_0y}\bigr)\nonumber\\
&=&\mathbb{P}\bigl(\tau_y>t, \tau_{x_0}>\tau^{x_0y}\bigr)+\mathbb{P}\bigl(\tau_{x_0}>t,\tau_{x_0}<\tau^{x_0y}\bigr)\nonumber\\
&\ge&\mathbb{P}\bigl(\tau_y>t, \tau_{x_0}>\tau^{x_0y}\bigr)+\mathbb{P}\bigl(\tau_y>t,\tau_{x_0}<\tau^{x_0y}\bigr)\nonumber\\
&=&\mathbb{P}\bigl(\tau_y>t\bigr).
\end{eqnarray*}
 \end{proof}
\subsubsection{Construction of the Coupling}
Without loss of generality we may assume that $a<x<y<b$ and $\mu > 0$. 





\begin{enumerate}
\item[I)]
We let both processes run in opposite directions, i.e we define $X_t=x+\mu t+\sigma B_t$ and $Y_t=y+\mu t-\sigma B_t$ up to the time $\tau^I=\tau_1\wedge\tau_2\wedge \tau_3$, where 
\begin{equation*}
\tau_1=\inf\{t>0\mid x+\mu t +\sigma B_t=a\rbrace, \,\tau_2=\inf \lbrace t> 0\mid y+\mu t -\sigma B_t=b\rbrace 
\end{equation*}
and 
\begin{equation*}
\tau_3 = \inf \lbrace t > 0 \mid x + \mu t + \sigma B_t = y +\mu t- \sigma B_t\rbrace.
\end{equation*}
If $\tau^{I} = \tau_3$, we are done. Otherwise, move to the next stage.
\item[II)] We set 
\begin{equation*}
\begin{split}
X_{\tau^I}=
\begin{cases}
x_0 &\text{ if $\tau^I = \tau_1$} \\
x + \mu \tau^I +\sigma B_{\tau^I} &\text{ if $\tau^I=\tau_2$}
\end{cases}
\end{split}
\end{equation*}
and 
\begin{equation*}
\begin{split}
Y_{\tau^I}=
\begin{cases}
y + \mu \tau^I - \sigma B_{\tau^I} &\text{ if $\tau^I = \tau_1$} \\
x_0 &\text{ if $\tau^I=\tau_2$}.
\end{cases}
\end{split}
\end{equation*}
By definition of $\tau^{I}$, we have that either $X_{\tau^{I}}= x_0 =\frac{a+b}{2}$ or $Y_{\tau^{I}}= x_0=\frac{a+b}{2}$ and hence $|X_{\tau^{I}}-Y_{\tau^{I}}| \le \frac{b-a}{2}$. For $t > \tau^{I}$ we define $(X_t)_{t}$ analogous to the way we did it in the representation \eqref{e:represI}. The process $(Y_t)_{t}$ is continued also analogous to the representation \eqref{e:represI} but with the difference that for $(Y_t)_{t\ge 0}$ we use $(-B_t)_{t\ge 0}$ instead of $(B_t)_{t\ge 0}$. For a more formal description we refer to the proof of Lemma \ref{l:domination}.

We wait until time $\tau^{II}$
\begin{equation*}
\begin{split}
\tau^{II}= \inf \bigl \lbrace t \geq \tau^I \mid |X_{t}-Y_{t}| \in \lbrace 0, (b-a)/2\rbrace\bigr\rbrace.
\end{split}
\end{equation*}
If $X_{\tau^{II}}=Y_{\tau^{II}}$, we are done. Otherwise, move to the next stage.
\item[III)]
Now if $X_{\tau^{II}}\not=Y_{\tau^{II}}$, it follows by definition that  
\begin{equation}\label{essential_part}
|X_{\tau^{II}}-Y_{\tau^{II}}|=\frac{b-a}{2},
\end{equation}
For $t\ge \tau^{II}$, we once again change the sign of the Brownian
motion in the definition of the process $Y$, i.e. we define for $t \geq \tau^{II}$
\begin{equation*}
X_t=X_{\tau^{II}}+\mu(t-\tau^{II})+\sigma (B_t-B_{\tau^{II}}) \text{  and  } Y_t=Y_{\tau^{II}}+\mu(t-\tau^{II})+\sigma (B_t-B_{\tau^{II}})
\end{equation*}
and wait until
\begin{displaymath}
\tau_{coup}=\tau^{III} = \inf \lbrace t \geq \tau^{II}\mid X_t = x_0 \rbrace.
\end{displaymath}
\item[IV)]
For $t\ge \tau_{coup}$, we obviously can choose $X_t=Y_t=\frac{a+b}{2}+\mu(B_t-B_{\tau_{coup}})+\sigma(B_t-B_{\tau_{coup}})$, which means that the processes have been coupled after $\tau_{coup}$.
\end{enumerate}

We note, that by construction the process $(X_t)_{t \geq 0}$ [$(Y_t)_{t\geq 0}$] is a $L^{\sigma,\mu}$-diffusion in $(a,b)$ with jump boundary and jump distribution $\delta_{x_0}$ starting in $x$ [$y$].

\subsubsection{Analysis of the Coupling.}
To finalize the proof, we have to show that
\begin{enumerate}
\item[$\bullet$]
$\tau_{coup}$ is a coupling time and $\lim_{t\rightarrow\infty}-\frac{1}{t}\log\mathbb{P}\bigl(\tau_{coupl}>t\bigr)\ge \frac{8\sigma^2\pi^2}{(b-a)^2}$.
\end{enumerate}
In order to prove this statement we first look at the time $\tau^{II}$ in somewhat more detail.

\begin{lemma}\label{l:domination}
Let $(X_t)_{t \geq 0}$ and $(Y_t)_{t \geq 0}$ be constructed as above in $I)$ and $II)$. and $\tau^{II}$ as defined above.
Then we have
\begin{equation}
\mathbb{P}\bigl(\tau^{II}-\tau^{I}>t)\le\mathbb{P}\bigl(B_s\not\in\{-\frac{b-a}{8\sigma},\frac{b-a}{8\sigma}\},\,\forall s\le t\bigr).
\end{equation}
\end{lemma}
\begin{proof}
Observe that $\forall\,t\in [\tau^I,\tau^{II}]$ we have 
$X_{t-}\neq a,\, Y_{t-}\neq a$,
i.e. for $\tau^I < t\le\tau^{II}$ jumps occur only via the boundary point $b$. Without loss of generality we may assume that $X_{\tau^{I}}=x_0$ and hence $Y_{\tau^{I}}>x_0$.
For $\tau^{I}\le t\le\tau^{II}$ let
\[
Y_t=\left\{\begin{array}{cc}
Y_{\tau^{I}}+\mu(t-\tau^{I})-\sigma(B_t-B_{\tau^{I}})&:\,\,\tau^{I}\le t<\tau^{I,2}\\
\frac{a+b}{2}+\mu(t-\tau^{I,2})-\sigma(B_t-B_{\tau^{I,2}})&:\,\,\tau^{I,2}\le t<\tau^{I,4}\\
\frac{a+b}{2}+\mu(t-\tau^{I,4})-\sigma(B_t-B_{\tau^{I,4}})&:\,\,\tau^{I,4}\le t<\tau^{I,6}\\
\cdots&\cdots\\
\frac{a+b}{2}+\mu(t-\tau^{I,n-2})-\sigma(B_t-B_{\tau^{I,n-2}})&:\,\,\tau^{I,n-2}\le t<\tau^{I,n}\\
\cdots &\cdots
\end{array}
\right.
\]
and similarly
\[
X_t=\left\{\begin{array}{cc}
\frac{a+b}{2}+\mu(t-\tau^{I})+\sigma(B_t-B_{\tau^{I}})&:\,\,\tau^{I}\le t<\tau^{I,3}\\
\frac{a+b}{2}+\mu(t-\tau^{I,3})+\sigma(B_t-B_{\tau^{I,3}})&:\,\,\tau^{I,3}\le t<\tau^{I,5}\\
\frac{a+b}{2}+\mu(t-\tau^{I,5})+\sigma(B_t-B_{\tau^{I,5}})&:\,\,\tau^{I,5}\le t<\tau^{I,7}\\
\cdots&\cdots\\
\frac{a+b}{2}+\mu(t-\tau^{I,n-1})+\sigma(B_t-B_{\tau^{I,n-1}})&:\,\,\tau^{I,n-1}\le t<\tau^{I,n+1}\\
\cdots &\cdots
\end{array}
\right.,
\]

where
\[
\tau^{I,2}=\inf\{t>\tau^{I}: Y_{\tau^{I}}+\mu (t-\tau^{I})-\sigma(B_t-B_{\tau^{I}})=b\}, 
\]

\[
\tau^{I,3}=\inf\{t>\tau^{I,2}: X_{\tau^{I}}+\mu (t-\tau^{I,2})+\sigma(B_t-B_{\tau^{I,2}})=b\}
\]
and inductively for all $k\in\mathbb{N}$ (note that after a jump, $Y_t-X_t$ changes its sign) 
\[
\tau^{I,2k}=\inf\{t>\tau^{I,2k-1}: Y_{\tau^{I,2k-2}}+\mu (t-\tau^{I,2k-2})-\sigma(B_t-B_{\tau^{I,2k-2}})=b\}
\]
and
\[
\tau^{I,2k+1}=\inf\{t>\tau^{I,2k}: X_{\tau^{I,2k}}+\mu (t-\tau^{I,2k})+\sigma(B_t-B_{\tau^{I,2k}})=b\}.
\]

We are interested in the differences 
$|Y_t-X_t|$ for $t\in[\tau^{I},\tau^{II})$.
For $t\in[\tau^{I},\tau^{I,2})$, we have that $\{t<\tau^{II}\}$ is equivalent to  
\[
|Y_t-X_t|=Y_t-X_t=Y_{\tau^{I}}-\frac{a+b}{2}-2\sigma (B_t-B_{\tau^{I}})\not\in\{0,\frac{b-a}{2}\},
\]
which again is equivalent to
\begin{equation}\label{the_distance}
B_t-B_{\tau^{I}}\in\left(-\frac{b-Y_{\tau^{I}}}{2\sigma},\frac{Y_{\tau^{I}}-\frac{a+b}{2}}{2\sigma}\right).
\end{equation}
Due to the definition of $X$ and $Y$ we have at $\tau^{I,2}-$ that 
\[
Y_{\tau^{II}-}-X_{\tau^{II}-}=   Y_{\tau^{I}}-\frac{a+b}{2}-2(\sigma B_{\tau^{I,2}}-B_{\tau^{I}})
\]
(note that the drift cancels out and the BM is continuous). At time $\tau^{I,2}$, the $Y$-process jumps to $\frac{a+b}{2}$ (the process jumps $\frac{b-a}{2}$ to the left), and hence we have
\[
|Y_{\tau^{I,2}}-X_{\tau^{I,2}}|=X_{\tau^{I,2}}-Y_{\tau^{I,2}}=\frac{b-a}{2}-\bigl(Y_{\tau^{I}}-\frac{a+b}{2}-2\sigma (B_{\tau^{I,2}}-B_{\tau^{I}})\bigr)=b-Y_{\tau^{I}}+2\sigma (B_{\tau^{I,2}}-B_{\tau^{I}}).
\]
For $t\in[\tau^{I,2},\tau^{I,3})$ we obtain by the definitions of $X_t$ and $Y_t$ that
\begin{eqnarray}
|X_t-Y_t|&=&(X_t-X_{\tau^{I,2}})-(Y_t-Y_{\tau^{I,2}})+(X_{\tau^{I,2}}-Y_{\tau^{I,2}})\nonumber\\
&=&2\sigma (B_t-B_{\tau^{I,2}})+b-Y_{\tau^{I}}+2\sigma (B_{\tau^{I,2}}-B_{\tau^{I}})\nonumber\\
&=&b-Y_{\tau^{I}}+2\sigma (B_t-B_{\tau^{I}}).
\end{eqnarray}
Hence, once again we see that
for $t\in[\tau^{I,2},\tau^{I,3})$, we have that $\{t<\tau^{II}\}$ is equivalent to \eqref{the_distance}.
For $t\in[\tau^{I,3},\tau^{I,4})$, these calculations repeat (having in mind that the sign of $Y_t-X_t$ changes after every jump), so we
end up with 
\begin{equation}
\mathbb{P}\bigl(\tau^{II}-\tau^{I}>t\bigr)=\mathbb{P}\biggl(B_s-B_{\tau^{I}}\in(-\frac{b-Y_{\tau^{I}}}{2\sigma},\frac{Y_{\tau^{I}}-\frac{a+b}{2}}{2\sigma})\,\,\forall s\in (\tau^{I},\tau^{I}+t)\biggr).
\end{equation} 
But then the claim follows from Proposition \ref{firstprop}.

\end{proof}

In order to finish the proof of Theorem \ref{t:main}, we write $\tau_{coup}=\tau^{I}+(\tau^{II}-\tau^{I})+(\tau_{coup}-\tau^{II})$.
We have
\begin{eqnarray}
\mathbb{P}\bigl(\tau_{coup}>t\bigr)&\le&\frac{\mathbb{E}\bigl[e^{\lambda\tau_{coup}}\bigr]}{e^{\lambda t}}\nonumber\\
&=&\frac{\mathbb{E}\bigl[e^{\lambda\tau^{I}}\mathbb{E}\bigl[e^{\lambda(\tau^{II}-\tau^{I})}|\mathcal{F}_{\tau^{I}}\bigr]\mathbb{E}\bigl[e^{\lambda(\tau_{coup}-\tau^{II})} |\mathcal{F}_{\tau^{II}}\bigr]\bigr]}{e^{\lambda t}},
\end{eqnarray}
where $\mathcal{F}_{\tau^{I}}$ denotes the $\sigma$-field generated by $B_s$ for $s\le\tau^{II}$.\\ 
First of all, note that by definition of $\tau^{I}$ we have that
\begin{equation}
\mathbb{P}\bigl(\tau^{I}>b-a\bigr)=0
\end{equation}
and therefore
\begin{eqnarray}\label{to_show}
\mathbb{P}\bigl(\tau_{coup}>t\bigr)\le C_{\lambda}^{(1)}\cdot\frac{\mathbb{E}\bigl[\mathbb{E}\bigl[e^{\lambda(\tau^{II}-\tau^{I})}|\mathcal{F}_{\tau^{I}}\bigr)\mathbb{E}\bigl[e^{\lambda(\tau_{coup}-\tau^{II})} |\mathcal{F}_{\tau^{II}}\bigr]\bigr]}{e^{\lambda t}}.
\end{eqnarray}
Let us determine upper bounds for $\mathbb{E}\bigl[e^{\lambda(\tau^{II}-\tau^{I})}|\mathcal{F}_{\tau^{I}}\bigr]$ and $\mathbb{E}\bigl[e^{\lambda(\tau_{coup}-\tau^{II})} |\mathcal{F}_{\tau^{II}}\bigr]$. 
By Lemma \ref{l:domination} we have 
\begin{eqnarray}
\mathbb{P}\bigl(\tau^{II}-\tau^{I}>t|\mathcal{F}_{\tau^{I}}\bigr) &\le&\mathbb{P}\biggl( B_s\not\in\{-\frac{b-a}{8\sigma},\frac{b-a}{8\sigma}\} \mbox{ for all }s\le t\biggr)\nonumber\\
&\le&C e^{-\frac{8\sigma^2\pi^2}{(b-a)^2}t}.
\end{eqnarray}
Hence partial integration yields
\begin{eqnarray}\label{before}
\mathbb{E}\bigl[e^{\lambda(\tau^{II}-\tau^{I})}|\mathcal{F}_{\tau^{I}}\bigr]&=&\int_{0}^{\infty}e^{\lambda s}\mathbb{P}\bigl(\tau^{II}-\tau^{I}\in ds|\mathcal{F}_{\tau^{I}}\bigr)\nonumber\\
&=&-e^{\lambda s}\mathbb{P}\bigl(\tau^{II}-\tau^{I}>s\bigr)\bigr|_{0}^{\infty}+\lambda\int_{0}^{\infty}e^{\lambda s}\mathbb{P}\bigl(\tau^{II}-\tau^{I}>s\bigr)ds\nonumber\\
&\le& C^{(2)}_{\lambda}<\infty\,\,\,\forall \lambda<\frac{8\sigma^2\pi^2}{(b-a)^2}.
\end{eqnarray}
In order to establish an upper bound for $\mathbb{E}\bigl[e^{\lambda(\tau_{coup}-\tau^{II})} |\mathcal{F}_{\tau^{II}}\bigr]$, first observe that after time $\tau^{II}$, both 
processes have distance $\frac{b-a}{2}$ (or have already been successfully coupled). Since the processes run parallel after $\tau^{II}$ in the 
interval $(a,b)$ and have distance $\frac{a-b}{2}$, one of the processes necessarily hits the boundary of the interval and then both processes immediately glue together. The boundary is hit by one of the processes exactly if a Brownian motion with variance $\sigma$ and drift $\mu$ exits an interval with length $\frac{b-a}{2}$. Hence, it follows from Lemma \ref{l:exittime} that for a suitable constant $K > 0$ we have
\begin{equation}
\mathbb{P}\bigl(\tau_{coup}-\tau^{II}>t |\mathcal{F}_{\tau^{II}}\bigr)\leq K\,e^{-\bigl(\frac{\pi^2\sigma^2}{2((b-a)/2)^2}+\frac{\mu^2}{2\sigma^2}\bigr)t}
\end{equation}
and hence by partial integration as in \eqref{before} 
\begin{equation}\label{second_eq}
\mathbb{E}\bigl[e^{\lambda(\tau_{coup}-\tau^{II})} |\mathcal{F}_{\tau^{II}}\bigr]<C^{(3)}_{\lambda}<\infty\,\,\,\forall \lambda < \frac{2\pi^2\sigma^2}{(b-a)^2}+
\frac{\mu^2}{2\sigma^2}.
\end{equation}
Now plugging \eqref{before} and \eqref{second_eq} into \eqref{to_show}, we finally end up with
\begin{equation}\label{e:tailofcoupl}
\mathbb{P}\bigl(\tau_{coup}>t\bigr)<C^{(4)}_{\lambda}e^{-\lambda t},\,\,C^{(4)}_{\lambda}<\infty\,\,\,\,\forall \lambda<\min\bigl(\frac{2\sigma^2\pi^2}{(b-a)^2}+\frac{\mu^2}{2\sigma^2},\frac{8\sigma^2\pi^2}{(b-a)^2}\bigr).
\end{equation}
This finalizes the proof.

We expect that the coupling presented above is efficient for all values of $\mu$. Equation \eqref{e:tailofcoupl} thus leads to the conjecture 
\begin{displaymath}
\mu_0(\sigma,x_0) = \sqrt{3}\frac{2\sigma^2\pi}{b-a},
\end{displaymath}
which we already have formulated in equation \eqref{e:threshold}.
\section*{Acknowledgement}
The authors would like to thank Ross Pinsky for helpful remarks concerning the topic of this work.


\begin{thebibliography}{99}
\bibitem{BaP1} Iddo Ben-Ari and Ross G. Pinsky, {Spectral analysis of a family of second-order elliptic operators with nonlocal boundary condition indexed by a probability measure},  \textit{J. Funct. Anal.} \textbf{251}  (2007), 122--140. 
\bibitem{BaP2} Iddo Ben-Ari and Ross G. Pinsky, {Ergodic behavior of diffusions with random jumps from the boundary}, 
\textit{Stochastic Process. Appl.} \textbf{119} (2009), 864--881. 
\bibitem{GK1} Ilie Grigorescu and Min Kang, {Brownian motion on the figure eight}, \textit{J. Theoret. Probab.} \textbf{15}  (2002), 817--844.
\bibitem{GK2} Ilie Grigorescu and Min Kang, {Ergodic properties of multidimensional Brownian motion with rebirth},  \textit{Electron. J. Probab.} \textbf{12}  (2007), no. 48, 1299--1322
\bibitem{GK3} Ilie Grigorescu and Min Kang, {The Doeblin condition for a class of diffusions with jumps}, preprint (2009)
\bibitem{KW} Martin Kolb and Achim W\"ubker, {On the spectral gap of Brownian motion with jump boundary}, submitted (2010)
\bibitem{K} Elena Kosygina, {Brownian flow on a finite interval with jump boundary conditions}, \textit{Discrete Contin. Dyn. Syst. Ser. B} \textbf{6} (2006), 867--880
\bibitem{LL} Wenbo V. Li and Yuk Leung, Fastest Rate of convergence for Brownian motion with jump boundary, preprint (2011)
\bibitem{LLR} Wenbo V. Li, Yuk J. Leung and Rakesh, {Spectral analysis of Brownian motion with jump boundary}, \textit{Proc. Americ. Math. Soc.} \textbf{136} (2008), 4427-4436
\bibitem{LP} Wenbo V. Li and Jun Peng, Diffusions with holding and jumping boundary, preprint (2011)
\bibitem{P95} Ross G. Pinsky, Positive Harmonic Functions and Diffusion, Cambridge University Press, 1995
\bibitem{W} Joachim Weidmann, Lineare Operatoren in Hilbertr\"aumen, Teil II: Anwendungen, Teubner Verlag, Wiesbaden, 2003 
\end{thebibliography}
\end{document}